\documentclass{amsart}

\usepackage{amsmath,amssymb,verbatim}
\usepackage{amsthm}
\usepackage{stmaryrd}
\usepackage{enumerate}
\usepackage{url}
\usepackage{wasysym}

\def\seq{\subseteq}
\def\inv{^{\text{-}1}}
\newcommand{\nv}{\text{-}}

\def\Th{\operatorname{Th}}

\newcommand{\abar}{\bar{a}}

\newcommand{\cbar}{\bar{c}}

\renewcommand{\hbar}{\bar{h}}

\newcommand{\mbar}{\bar{m}}
\newcommand{\nbar}{\bar{n}}

\newcommand{\rbar}{\bar{r}}
\newcommand{\ubar}{\bar{u}}
\newcommand{\vbar}{\bar{v}}

\newcommand{\xbar}{\bar{x}}
\newcommand{\ybar}{\bar{y}}

\newcommand{\ku}{\mathfrak{u}}
\newcommand{\ms}{\mathfrak{s}}

\newcommand{\kn}{\mathfrak{n}}

\newcommand{\de}{\mathbin{\dot{=}}}
\def\mod{\operatorname{mod}}
\newcommand{\mand}{\makebox[.4in]{and}}
\newcommand{\noit}[1]{\noindent\textit{#1}}
\newcommand{\claim}{\hfill$\dashv_{\text{\scriptsize{claim}}}$}

\newcommand{\cA}{\mathcal{A}}
\newcommand{\cB}{\mathcal{B}}
\newcommand{\cC}{\mathcal{C}}

\newcommand{\cL}{\mathcal{L}}
\newcommand{\cM}{\mathcal{M}}
\newcommand{\cN}{\mathcal{N}}

\newcommand{\cP}{\mathcal{P}}
\newcommand{\cQ}{\mathcal{Q}}
\newcommand{\cR}{\mathcal{R}}

\newcommand{\cZ}{\mathcal{Z}}

\def\C{\mathbb C}

\def\N{\mathbb N}
\def\P{\mathbb P}
\def\Q{\mathbb Q}
\def\R{\mathbb R}

\def\Z{\mathbb Z}

\newtheorem{theorem}{Theorem}[section]
\newtheorem{lemma}[theorem]{Lemma}
\newtheorem{corollary}[theorem]{Corollary}
\newtheorem{proposition}[theorem]{Proposition}

\newtheorem{fact}[theorem]{Fact}

\theoremstyle{definition}
\newtheorem{definition}[theorem]{Definition}
\newtheorem{example}[theorem]{Example}
\newtheorem{remark}[theorem]{Remark}
\newtheorem{question}[theorem]{Question}
\newtheorem{notation}[theorem]{Notation}

\title[Multiplicative structure in stable expansions of $(\Z,+,0)$]{Multiplicative structure in stable expansions of the group of integers}

\author{Gabriel Conant}

\address{Department of Mathematics\\
University of Notre Dame\\
Notre Dame, IN, 46656, USA}

\email{gconant@nd.edu}

\begin{document}

\begin{abstract}
We define two families of expansions of $(\Z,+)$ by unary predicates, and prove that their theories are superstable of $U$-rank $\omega$. The first family consists of expansions $(\Z,+,A)$, where $A$ is an infinite subset of a finitely generated multiplicative submonoid of $\N$. Using this result, we also prove stability for the expansion of $(\Z,+)$ by all unary predicates of the form $\{q^n:n\in\N\}$ for some $q\in\N_{\geq 2}$. The second family consists of sets $A\seq\N$ which grow asymptotically close to a $\Q$-linearly independent increasing sequence $(\lambda_n)_{n=0}^\infty\seq\R^+$ such that $\{\frac{\lambda_n}{\lambda_m}:m\leq n\}$ is closed and discrete.
\end{abstract}


\maketitle

\section{Introduction}\label{sec:intro}
\setcounter{theorem}{0}
\numberwithin{theorem}{section}

A common theme in model theory is to fix a mathematical structure, whose definable sets satisfy certain tameness properties, and study expansions of that structure by new definable sets which preserve tameness. Perhaps the most widely known example of this kind of program is the study of \emph{o-minimal} expansions of the real ordered field $\cR=(\R,+,\cdot,<)$, where a particularly celebrated result, due to Wilkie \cite{Wilkie}, is that the expansion of $\cR$ by the exponential function remains o-minimal.  The study of o-minimal expansions of $\cR$ has been used in applications to Diophantine geometry including a new proof of the Manin-Mumford conjecture, and cases of the Andr\'{e}-Oort conjecture (further details can be found in \cite{ScanAO}). Another example is the study of \emph{stable} expansions of the complex field $\cC=(\C,+,\cdot)$. In \cite{BeZi}, Belegradek and Zilber study expansions of $\cC$ by finitely generated multiplicative subgroups of the unit circle. Expansions of $\cC$ by arbitrary finitely generated multiplicative subgroups are analyzed by Van den Dries and G\"{u}nayd{\i}n in \cite{vdDGu}. This work relates to study of finitely generated subgroups of $A(\C)$, where $A$ is a semiabelian variety over $\C$. In model-theoretic language, the  Mordell-Lang conjecture, proved by Faltings, is that the induced structure on any such subgroup is stable and $1$-based (further details can be found in \cite{PiLC}).

In this article, we focus on expansions of the group of integers $\cZ=(\Z,+)$. The group $\cZ$ is a well-known example of a superstable group, whose behavior has motivated the rich model theoretic study of modules and $1$-based (or weakly normal) stable groups, going back to the 1980s (see \cite{HrPi1B}, \cite{prestbook}). On the other hand, when compared to the breadth and depth of results on $\cR$ and $\cC$, much less is known about tame expansions of $\cZ$. The most well understood example of a \emph{proper} expansion of $\cZ$ is the ordered group $(\Z,+,<)$, often called \emph{Presburger arithmetic}. While $(\cZ,+,<)$ is not o-minimal, it does satisfy other model theoretic notions of ``minimality" such as \emph{quasi-o-minimality} and \emph{dp-minimality}. However, results by several authors (e.g. \cite{ADHMS1}, \cite{BPW}) have shown that $(\Z,+,<)$ has no proper expansions satisfying these tameness properties (or even several weaker properties, e.g. \cite{DoGo}). Thus, for expansions of $(\Z,+,<)$, the best possible hope for a nontrivial research program would seem to be in the setting of NIP theories.  For instance, in \cite{PointPA}, Point obtains striking quantifier elimination results for expansions of $(\Z,+,<)$ by various \emph{sparse} sequences, which allows one to conclude the expansions are NIP \cite{ADHMS2}. 

Being a totally ordered structure, $(\Z,+,<)$ is not stable. A surprising fact is that, after its canonization in the 1980's as a foundational example of a stable group, there was no known example of a proper stable expansion of $\cZ$. This remained the case until 2014, when Palac\'{i}n-Sklinos \cite{PaSk} and Poizat \cite{PoZ} independently gave the first examples, which included the expansions $(\Z,+,q^{\N})$ where $q\in\N_{\geq 2}$ and $q^{\N}=\{q^n:n\in\N\}$. The examples in \cite{PaSk} and \cite{PoZ} were generalized by the author \cite{CoSS} to the class of \emph{geometrically sparse} subsets of $\Z$, which is a robust class of sets whose growth rate is ``at least geometric" in a fairly strong sense (see Definition \ref{def:gs}).  Independently of \cite{CoSS}, the examples from \cite{PaSk} and \cite{PoZ} were also generalized by Lambotte and Point \cite{PoLa}, who prove stability for certain families of expansions of $\cZ$ overlapping nontrivially with those studied in \cite{CoSS}. In many of these examples of proper stable expansions of $\cZ$ by a single subset $A\seq\Z$, the set $A$ is inherently \emph{multiplicative}  in the sense that its asymptotic behavior is dominated by the powers of a fixed real number $\lambda>1$. A specific example is the Fibonacci sequence, $(F_n)_{n=0}^\infty$, which satisfies $|F_n-\frac{1}{\sqrt{5}}\phi^n|\leq 1$ for all $n\in\N$, where $\phi$ is the golden ratio. 

In this article, we generalize this multiplicative nature of proper stable expansions of $\cZ$ in two broad ways. First, we view the original examples of Palac\'{i}n-Sklinos and Poizat as expansions of $\cZ$ by cyclic multiplicative submonoids of $\Z^+$. Our first result is the following generalization.

\newtheorem*{thm:CMS}{Theorem \ref{thm:CMS}}
\begin{thm:CMS}
Let $\Gamma$ be a finitely generated multiplicative submonoid of $\Z^+$. If $A\seq \Gamma$ is infinite then $(\Z,+,A)$ is superstable of $U$-rank $\omega$.
\end{thm:CMS}

So, one the one hand, we find a commonality to the study of stable expansions of $\cC=(\C,+,\cdot)$ and the work of Van den Dries and G\"{u}nayd{\i}n \cite{vdDGu} mentioned above. Indeed, the main non-model-theoretic tool in the proof of Theorem \ref{thm:CMS} is a result from algebraic number theory, due to Evertse, Schlickewei, and Schmidt \cite{ESS}, which gives bounds on the number non-degenerate solutions to linear equations in finitely generated multiplicative subgroup of $(\C^*,\cdot)$ (see Fact \ref{fact:ESS}). There are many results of this kind, going back to Schmidt's Subspace Theorem \cite{SchST}, and this behavior in multiplicative groups has many names (e.g. the \emph{Mann property} or  \emph{Mordell-Lang property} in \cite{vdDGu}; and \emph{Lang type} in \cite{PiLC}). On the other hand, there is a stark difference in Theorem \ref{thm:CMS} in that a stable expansion can be obtained using \emph{any arbitary subset} of $\Gamma$. This is not the case in expansions of $\cC$ since, for example, $(\C,+,\cdot,2^{\Z})$ is stable by \cite{vdDGu}, while $(\C+,\cdot,2^{\N})$ is unstable since the ordering on $2^{\N}$ is definable. Stability for the expansion of $\cZ$ by arbitrary subsets of $\Gamma$ also allows us to obtain new examples of stable expansions of $\cZ$ by many unary predicates.

\newtheorem*{thm:MUP}{Theorem \ref{thm:MUP}}
\begin{thm:MUP}
For any integers $q_1,\ldots,q_d\geq 2$, $(\Z,+,q_1^{\N},\ldots,q_d^{\N})$ is superstable of $U$-rank $\omega$. Therefore $(\Z,+,(q^{\N})_{q\geq 2})$ is stable.
\end{thm:MUP}

Theorem \ref{thm:CMS} yields another new phenomenon in stable expansions of $\cZ$. In particular, call a set $A\seq\Z^+$ \emph{lacunary} if $\limsup_{n\to\infty}\frac{a_{n+1}}{a_n}>1$, where $(a_n)_{n=0}^\infty$ is an increasing enumeration of $A$. The sets considered in \cite{CoSS}, \cite{PoLa}, \cite{PaSk}, and \cite{PoZ} are all lacunary (although there are examples where the $\limsup$ is as close to $1$ as desired). So this motivates the question, asked in \cite{CoSS} and \cite{PoLa}, of whether there is a stable expansion of $\cZ$ by a non-lacunary subset of $\Z^+$. Theorem \ref{thm:CMS} produces such examples since, by Furstenburg \cite{FurstLac}, a finitely generated submonoid $\Gamma$ of $\Z^+$ is non-lacunary whenever there are $a,b\in\Gamma$ such that $\log_b a$ is irrational. 

On the other hand, all known \emph{unstable} expansions of the form $(\Z,+,A)$, with $A\seq\Z^+$, satisfy $\liminf_{n\to\infty}\frac{a_{n+1}}{a_n}=1$ (e.g. the primes or perfect squares). In Theorem \ref{thm:badex}, we give an example of an unstable expansion $(\Z,+,A)$ such that $\lim_{n\to\infty}\frac{a_{n+1}}{a_n}$ exists and is strictly greater than $1$ (in particular $A_q:=\{q^n+n:n\in\N\}$ where $q\geq 2$). This example provides new information about the asymptotic behavior of sets $A\seq\Z^+$, which produce stable expansions of $\cZ$. In \cite{CoSS}, this is formulated using the notion of a \emph{geometric sequence}, which we define to be a strictly increasing sequence $(\lambda_n)_{n=0}^\infty$ in $\R^+$ such that $\{\frac{\lambda_m}{\lambda_n}:n\leq m\}$ is closed and discrete (e.g. $\lambda_n=b^n$ for some $b\in\R^{>1}$). The main result of \cite{CoSS} is that, for $A=(a_n)_{n=0}^\infty\seq\Z$, if $|a_n-\lambda_n|$ is $O(1)$ for some geometric sequence $(\lambda_n)_{n=0}^\infty$, then $(\Z,+,A)$ is stable. While this notion of a geometric sequence is a robust way to describe the nature of stable expansions of $\cZ$, the examples in \cite{PoLa} show that $O(1)$ can be relaxed in some cases. On the other hand, the example in Theorem \ref{thm:badex} shows that $O(1)$ cannot even be weakened to $O(n)$ in general. Thus, our last main result is that $O(1)$ can be weakened substantially if we impose further restrictions on the geometric sequence.

\newtheorem*{thm:IS}{Theorem \ref{thm:IS}}
\begin{thm:IS}
Fix $B=(b_n)_{n=0}^\infty\seq\Z$ and a $\Q$-linearly independent geometric sequence $(\lambda_n)_{n=0}^\infty$ such that $|b_n-\lambda_n|$ is $o(\lambda_n)$. For any finite $F\seq\Z$ and infinite $A\seq B+F$, $(\Z,+,A)$ is superstable of $U$-rank $\omega$.
\end{thm:IS}

A concrete family of examples covered by this theorem, which are straightforward but still illustrate the flexibility of the statement, is as follows. Fix algebraically independent reals $\tau_1,\ldots,\tau_k>1$ and construct a sequence $(\lambda_n)_{n=0}^\infty$ by setting $\lambda_0=\tau_1$ and $\lambda_{n+1}=c_n\tau_{i_n}\lambda_n$, where $c_n\in\Z^+$ and $i_n\in\{1,\ldots,k\}$ are arbitrary. Finally, pick $g\colon \N\to\R^+$ such that $g(n)$ is $o(\tau^n)$. Then the set $B=\{\llbracket\lambda_n+g(n)\rrbracket:n\in\N\}$ satisfies the assumptions of Theorem \ref{thm:IS} (where $\llbracket\cdot\rrbracket$ denotes integer part).  This theorem also generalizes a result of \cite{PoLa}, which covers the case when $A=B$ is eventually periodic modulo any fixed $m>0$, and $\lambda_n=\alpha\tau^n$ for some real number $\alpha>0$ and transcendental $\tau>1$.

To prove the theorems above, we use the same strategy from \cite{CoSS}, \cite{PoLa}, and \cite{PaSk} for showing that an expansion of $(\Z,+)$ by some fixed unary predicate is stable. Loosely speaking, we apply work of Casanovas and Ziegler \cite{CaZi} to show that stability of the expansion reduces to stability of the induced structure on the new predicate, and then we show that this induced structure is interpretable in a more familiar stable structure (see, e.g., Theorem \ref{thm:FG0}). This general strategy is explored in \cite{CoSS} in the setting of expansions of $U$-rank $1$ structures (e.g. $(\Z,+)$), and by quoting the work done there we will circumvent most model-theoretic considerations in the proofs (see Section \ref{sec:background} and, especially, Corollary \ref{cor:mainstab}).

\subsection*{Acknowledgements} I would like to thank John Baldwin, Fran\c{c}oise Point, and Pierre Simon for fruitful discussion; and also the referee for remarks and corrections.  Special thanks to Chris Laskowski for several helpful suggestions and comments on an earlier draft.

\section{Induced structure and stability}\label{sec:background}

In this section, we summarize the strategy (briefly outlined in Section \ref{sec:intro}) for proving stability of structures of the form $(\Z,+,A)$, where $A\seq\Z$. For background on basic first-order logic, including the notions of a first-order structure and definable sets in such a structure, we refer the reader to \cite{Mabook}.  For background on stability in first-order model theory, see \cite{Pibook}. For the reader unfamiliar with this topic, we briefly, but emphatically, say that stability is one of the first and most important notions of ``tameness" in first-order structures, and the properties found in stable structures have motivated most of modern research in the field of model theory. For instance, in stable structures, one finds a kind of paradise of good behavior, including a canonical notion of independence and dimension for definable sets, as well as any hope of classifying elementarily equivalent structures in higher cardinalities. 

\begin{definition}
Given a first order $\cL$-structure $\cM$, with universe $M$, let $\cL^{\cM}$ be the relational language consisting of, for any $n\geq 1$, an $n$-ary relation $R_X(\xbar)$ for every $\cM$-definable $X\seq M^n$. Given $A\seq M$, the \textbf{$\cM$-induced structure on $A$}, denoted $A^{\cM}$ is $\cL^{\cM}$-structure with universe $A$ such that, for each $\cM$-definable $X\seq M^n$, the relation $R_X(\xbar)$ is interpreted in $A^{\cM}$ as $A^n\cap X$. 
\end{definition}

Next, we set notation for expansions of $\cZ=(\Z,+,0)$. 

\begin{definition}
Given a sequence $(A_i)_{i\in I}$ of subsets of $\Z$, let $\cZ(A_i)_{i\in I}$ denote the expansion of $(\Z,+)$ obtained by adding a unary predicate for each $A_i$.
\end{definition}

\begin{remark}\label{rem:Zdef}
We will deal with many structures of the form $A^{\cZ}$, where $A\seq\Z$. Thus we recall that $X\seq\Z^n$ is $\cZ$-definable  if and only if it is in the Boolean algebra generated by cosets of subgroups of $\Z^n$. This follows from Presburger's work on $\Th(\cZ)$ (see, e.g., \cite[Lemma 1.9 \& Fact 1.10]{IbKiTa}). 
\end{remark}

\begin{definition}\label{def:SS}\cite{CoSS}
Fix $A\seq\Z$.
\begin{enumerate}
\item Given $n\geq 0$, set $\Sigma_n(A)=\{a_1+\ldots+a_k:k\leq n,~a_1,\ldots,a_k\in A\}$.
\item $A$ is \textbf{sufficiently sparse} if, for all $n\geq 0$, $\Sigma_n(\pm A)$ does not contain a nontrivial subgroup of $\Z$ (where $\pm A:=\{x\in\Z:|x|\in A\}$).
\end{enumerate}
\end{definition}

For sufficiently sparse sets, stability of $\cZ(A)$ is intimately tied to stability of $A^{\cZ}$, as detailed by the following fact. This result, which relies heavily on \cite{CaZi}, is proved by adapting techniques in \cite[Section 2]{PaSk}, which considers the case $A=q^{\N}$. The full proof is given in \cite[Theorems 2.11 \& 4.5]{CoSS}. 

\begin{fact}\label{fact:CZstable}
Suppose $A\seq\Z$ is sufficiently sparse. Then $\cZ(A)$ is stable if and only if $A^{\cZ}$ is stable. Moreover, $U(\cZ(A))\leq U(A^{\cZ})\cdot\omega$.
\end{fact}

In the above statement, the $U$-rank of a structure is ordinal valued and the expression $U(A^{\cZ})\cdot\omega$ refers to standard multiplication of ordinals. We will not need to deal directly with the definition of $U$-rank in this paper. For the reader unfamiliar with $U$-rank, we recall that in a stable first-order structure $\cM$, the geometry of definable sets is controlled by an abstract notion of independence called ``nonforking", which generalizes linear independence in vector spaces and algebraic independence in fields. One way to measure the complexity of nonforking independence in stable structures is with a rank function on types called $U$-rank; and a structure is \emph{superstable} when the rank of types in one variable is uniformly bounded by an ordinal. 
Fact \ref{fact:CZstable} provides an upper bound on the $U$-rank of $\cZ(A)$ in terms of the $U$-rank of $A^{\cZ}$. For the sets $A$ considered here, the $U$-rank of $A^{\cZ}$ will always be $1$ (see Proposition \ref{prop:monstab}), and so this will yield an upper bound of $\omega$ for the $U$-rank of the structures $\cZ(A)$ considered in this paper. To see that the rank is exactly $\omega$, we will apply the following result of Palac\'{i}n and Sklinos \cite{PaSk}. 

\begin{fact}\label{fact:PSexp}
$(\Z,+)$ has no proper stable expansions of finite $U$-rank.
\end{fact}

We will apply this fact several times in conclude that certain expansions of $\cZ$ are proper expansions. In each case, the expansion in question will be proper because of the following observation, which is immediate from Remark \ref{rem:Zdef}. 

\begin{proposition}
Suppose $A\seq\Z$ is infinite and either bounded below or bounded above. Then $A$ is not definable in $\cZ$.
\end{proposition}

We now set some terminology concerning expansions and reducts requiring a change of universe. Throughout the paper, \emph{definable} always means with parameters.

\begin{definition}
Let $\cM$ and $\cN$ be structures with universes $M$ and $N$, respectively. Then $\cM$ is a \textbf{virtual reduct} of $\cN$ if there is a bijection $f\colon M\to N$ such that the image under $f$ of any $\cM$-definable set is $\cN$-definable. In this case, $\cN$ is a \textbf{virtual expansion} of $\cM$. If $\cM$ is a virtual reduct of $\cN$ via $f$, and $\cN$ is a virtual reduct of $\cM$ via $f\inv$, then $\cM$ and $\cN$ are \textbf{interdefinable via $f$}. When $M=N$ and $f$ is the identity, we omit ``virtual", ``virtually", and ``via $f$" from the previous notions. 
\end{definition}

Throughout the  the paper, we will use the following notation.

\begin{notation}
Given an integer $n\geq 1$, let $\equiv_n$ denote the binary relation on integers given by equivalence modulo $n$, and let $[n]:=\{1,\ldots,n\}$. Given complex numbers  $z_1,\ldots,z_n,w\in\C$, we write $z_1+\ldots+z_n\de w$ if $z_1+\ldots+z_n=w$ and $\sum_{i\in I}z_i\neq 0$ for all proper nonempty $I\seq[n]$.  
\end{notation}

For the subsets $A\seq\N$ shown to be stable in \cite{CoSS}, \cite{PoLa}, and \cite{PaSk}, most of the difficulty in analyzing $A^{\cZ}$ comes from the induced structure from linear equations. 

\begin{definition}
Given $A\seq\Z$, let $A^{\cZ}_0$ be the reduct of $A^{\cZ}$ to relations of the form $\{\abar\in A^k: c_1a_1+\ldots+c_ka_k\de r\}$ where $k\geq 1$, $r\in\Z$, and $\cbar\in\{\nv 1,1\}^k$. 
\end{definition}

\begin{proposition}\label{prop:cong}
Given $A\seq\Z$, $A^{\cZ}$ is interdefinable with the expansion of $A^{\cZ}_0$ by unary predicates of the form $A\cap(n\Z+r)$ for $0\leq r<n$.
\end{proposition}
\begin{proof}
Most of this is implicit in \cite{PaSk}, \cite{PoLa}, and \cite{CoSS}, and so we only sketch the argument. Let $\cA$ be the expansion of $A^{\cZ}_0$ by unary predicates for sets of the form $A\cap (n\Z+r)$, for $0\leq r<n$. By the characterization of definable sets in $\cZ$ from Remark \ref{rem:Zdef}, we can describe $A^{\cZ}$ as the structure with universe $A$ and relations for sets of the form $A^k\cap X$, where $X\seq\Z^k$ is of one of the following forms:
\begin{enumerate}[$(i)$]
\item $X=\{\abar\in A^k:c_1a_1+\ldots+c_ka_k=r\}$ where $k\geq 1$, $r\in\Z$, and $\cbar\in\Z^k$; or
\item $X=\{\abar\in A^k:c_1a_1+\ldots+c_ka_k\equiv_n r\}$ where $k\geq 1$, $0\leq r<n$, and $\cbar\in\Z^k$.
\end{enumerate}
We want to show any such set is definable in $\cA$. Clearly (since we have equality in the language), we may assume $\cbar\in\{\nv 1,1\}^k$ (this reduction is mainly done for convenience later). It is also straightforward to show that any set $X$ of type $(ii)$ is definable from the unary predicates $A\cap (n\Z+r)$ for $0\leq r<n$ (see also \cite[Proposition 5.2]{CoSS}). So suppose $X$ is type $(i)$. Then 
\begin{multline*}
X=\{\abar\in A^k:c_1a_1+\ldots+c_ka_k\de r\}\\
\cup\bigcup_{\emptyset\neq I\subsetneq [k]}\left\{\abar\in A^k:\sum_{i\in I}c_ia_i=0\text{ and } \sum_{i\in [k]\backslash I} c_ia_i=r\right\}.
\end{multline*}
Working by induction on $k$, we see that $X$ is definable in $\cA$.
\end{proof}

For us, the key point of the previous proposition is that, given $A\seq\Z$, $A^{\cZ}$ is an expansion of $A^{\cZ}_0$ by \emph{unary} predicates. In general, expanding a structure by unary predicates can certainly affect stability of that structure (e.g. the expansion of $\cZ$ by $\N$ is unstable). However, for the sets $A\seq\Z$ considered here, we will show that the induced structure $A^{\cZ}_0$ has the property that \emph{any} expansion of $A^{\cZ}_0$ by unary predicates is stable of $U$-rank $1$. This will be extremely useful, as it will allow us to entirely ignore the extra induced structure in $A^{\cZ}$ from predicates of the form $A\cap(n\Z+r)$ for $0\leq r<n$. Even more, this will allow us to conclude that, for any $B\seq A$, $B^{\cZ}$ is also stable. This is formalized by Proposition \ref{prop:IRsub} below.

\begin{definition}
A structure $\cM$ is \emph{monadically stable of finite $U$-rank} if the expansion of $\cM$ by unary predicates for all subsets of $M$ is stable of finite $U$-rank.
\end{definition}

\begin{proposition}\label{prop:IRsub}
Given $A\seq\Z$, if $A^{\cZ}_0$ is a virtual reduct of a structure, which is monadically stable of finite $U$-rank, then, for any $B\seq A$, $B^{\cZ}$ is monadically stable of finite $U$-rank. 
\end{proposition}
\begin{proof}
Fix $A$ as in the statement, and $B\seq A$. Let $\cA$ be the expansion of $A^{\cZ}_0$ by unary predicates for all subsets of $A$, and let $\cB$ be the same for $B$. Then $\cA$ is stable of finite $U$-rank, $B$ is definable in $\cA$, and $\cB$ is the $\cA$-induced structure on $B$. By Proposition \ref{prop:cong}, $B^{\cZ}$ is a reduct of $\cB$.
\end{proof}

Altogether, this gives us the main conclusion of this section. 

\begin{corollary}\label{cor:mainstab}
Suppose $A\seq\Z$ is sufficiently sparse and $A^{\cZ}_0$ is a virtual reduct of a structure, which is monadically stable of finite $U$-rank. For any $B\seq A$, if $B$ is not definable in $(\Z,+)$ then $\cZ(B)$ is superstable of $U$-rank $\omega$.
\end{corollary}
\begin{proof}
Combine Fact \ref{fact:CZstable}, Fact \ref{fact:PSexp}, and Proposition \ref{prop:IRsub}, along with the observation that subsets of sufficiently sparse sets are sufficiently sparse.
\end{proof}

\section{Multiplicatively generated sets}\label{sec:MGG}

The goal of this section is our first main theorem.

\begin{theorem}\label{thm:CMS}
Suppose $\Gamma$ is a finitely generated multiplicative submonoid of $\Z^+$, and $A\seq\Gamma$ is infinite. Then $\cZ(A)$ is superstable of $U$-rank $\omega$.
\end{theorem}

For $A=q^{\N}$, stability of $\cZ(A)$ was first shown by Palac\'{i}n-Sklinos \cite{PaSk} and Poizat \cite{PoZ}. As $2^{\N}$ is geometrically sparse (see Definition \ref{def:gs}), the generalization to subsets $A\seq 2^{\N}$ follows from \cite{CoSS}. We will prove Theorem \ref{thm:CMS} using the strategy described in Section \ref{sec:background}. Specifically, to prove some expansion $\cZ(A)$ is stable, we first show that $A$ is sufficiently sparse, and then we interpret the $\cZ$-induced structure on $A$.

\subsection{Multiplicatively generated sets are sufficiently sparse}

\begin{definition}
A subset $\Gamma\seq\C^*$ has the \textbf{weighted sum property} if for all $k\geq 1$ there is some $n\geq 0$ such that, for any $r\in\Z^+$, the equation $x_1+\ldots+x_k\de r$ has at most $n$ solutions in $(\pm \Gamma)^k$. 
\end{definition}

Observe that the notation $\Sigma_n(A)$ generalizes directly to sets $A\seq \C$. The proof of the next lemma, which is essentially an application of Van der Waerden's Theorem, is a generalization of \cite[Lemma 4]{JaNa} (although the proof is more or less the same). 

\begin{lemma}\label{lem:VDW}
If $\Gamma\seq\C^*$ has the weighted sum property then for all $n\geq 1$, $\Sigma_n(\pm \Gamma)$ does not contain arbitrarily large arithmetic progressions whose common difference is a positive integer.
\end{lemma}
\begin{proof}
 Given $k\geq 1$ and $r>0$, let $\Gamma(k,r)\seq (\pm \Gamma)^k$ be the set of solutions to $x_1+\ldots+x_k\de r$. We prove the claim by induction on $n$. Suppose $n=1$. For any arithmetic progression $(z+td)_{t=0}^k$ in $\Sigma_1(\pm\Gamma)=\pm\Gamma$, with $d\in\Z^+$, the tuple $(z+td,\nv(z+(t-1)d))$ is in $\Gamma(2,d)$ for all but at most two indices $t\in [k]$. Therefore $\pm \Gamma$ does not contain arbitrarily large arithmetic progressions by assumption. Fix $n\geq 2$ and suppose we have the result for $n-1$. Let $m$ be a bound on the size of an arithmetic progression in $\Sigma_{n-1}(\pm \Gamma)$. Given $r\in\Z^+$, let $X(r)$ be the set of elements of $\pm\Gamma$, which appear in an element of $\Gamma(p,r)$ for some $p\leq 2n$. By assumption, there is an integer $c\geq 1$ such that $|X(r)|\leq c$ for all $r\in\Z^+$. 
 
 Recall that, given integers $u,v\geq 1$, the Van der Waerden number $W(u,v)$ is the smallest integer $N$ such that any coloring of $[N]$ by $u$ colors contains a monochromatic arithmetic progression of length $v$. Let $k=W(2c,m+1)$, and suppose $(z+td)_{t=0}^k$ is an arithmetic progression in $\Sigma_n(\pm \Gamma)$, with $d\in\Z^+$. For each $0\leq t\leq k$, fix $\sigma_t\in (\pm \Gamma)^{\leq n}$ such that $\sum \sigma_t=z+td$. For $1\leq t\leq k$ we have,
\[
d=z+td-(z+(t-1)d)=\sum \sigma_t+\sum\nv\sigma_{t-1},
\]
and thus we may find a subtuple $\mu_t$ of $(\sigma_t,\nv\sigma_{t-1})$ such that $\mu_t\in \Gamma(p,d)$ for some $1\leq p\leq 2n$. By construction, $\mu_t\in X(d)^{\leq 2n}$. Let $C=X(d)\times\{0,1\}$ and let $f\colon [k]\to C$ such that $f(t)_1$ is the first coordinate of $\mu_t$ and $f(t)_2=0$ if and only if $f(t)_1\in\sigma_t$. By choice of $k$, there is $(x,\epsilon)\in C$ and an arithmetic progression $P\seq[k]$ of length $m+1$ such that $f(t)=(x,\epsilon)$ for all $t\in P$. After exchanging $x$ with $\nv x$ and $P$ with $P-1$ if necessary, we obtain $x\in\pm \Gamma$ and an arithmetic progression $P\seq\{0,1,\ldots,k\}$ of length $m+1$ such that $x\in\sigma_t$ for all $t\in P$. Then $(z-x+td)_{t\in P}$ is an arithmetic progression of length $m+1$ in $\Sigma_{n-1}(\pm \Gamma)$, which is a contradiction. 
\end{proof}

The following fact is a widely used and extremely versatile result from algebraic number theory concerning solutions to linear equations in multiplicative subgroups of $\C^*$. This has been an active field of study with contributions from many people, and the following result of Evertse, Schlickewei, and Schmidt \cite{ESS} is in many ways a culmination of this work. In the statement of the result, we say that $\xbar,\ybar\in \C^k$ are \emph{scalar equivalent} if there is some $\lambda\in\C^*$ such that $x_i=\lambda y_i$ for all $i\in [k]$.

\begin{fact}\label{fact:ESS}
Suppose $\Gamma$ is a finitely generated multiplicative subgroup of $\C^*$. Then for all $k\geq 0$ there is some $n>0$ such that, for any $\alpha_1,\ldots,\alpha_k\in \C^*$, 
\begin{enumerate}[$(i)$]
\item $\alpha_1x_1+\ldots+\alpha_kx_k\de 0$ has at most $n$ solutions in $\Gamma^k$ modulo scalar equivalence, and
\item for all $\beta\in\C^*$,  $\alpha_1x_1+\ldots+\alpha_kx_k\de \beta$ has at most $n$ solutions in $\Gamma^k$.
\end{enumerate}
In particular, $\Gamma$ has the weighted sum property by $(ii)$. 
\end{fact}

In fact, this is only a weak version of the theorem in \cite{ESS}, in which $\C$ can be replaced by any algebraically closed field $K$ of characteristic $0$, $\Gamma$ is a finite rank multiplicative subgroup of $(K^*)^k$, and $n$ is effective and depends only on $k$ and the rank of $\Gamma$.

From Lemma \ref{lem:VDW} and Fact \ref{fact:ESS} we obtain the following conclusion.

\begin{corollary}\label{cor:mgss}
Suppose $A\seq\Sigma_k(\pm \Gamma)\cap\Z$, where $\Gamma$ is a finitely generated multiplicative subgroup of $\C^*$ and $k\geq 1$. Then, for all $n\geq 1$, $\Sigma_n(\pm A)$ does not contain arbitrarily large arithmetic progressions. In particular, $A$ is sufficiently sparse.
\end{corollary}

It follows that any finitely generated multiplicative submonoid of $\Z^+$ is sufficiently sparse.

\begin{remark}\label{rem:LHRR}
Corollary \ref{cor:mgss} can be used to show that many sets given by recurrence relations are sufficiently sparse. Specifically, suppose $A\seq\Z$ is enumerated by a linear homogeneous recurrence relation $(a_n)_{n=0}^\infty$, with constant coefficients and characteristic roots $\lambda_1,\ldots,\lambda_d\in\C^*$, each of multiplicity $1$. Then there are $\alpha_1,\ldots,\alpha_d\in\C^*$ such that $a_n=\alpha_1\lambda^n_1+\ldots+\alpha_d\lambda^n_d$. Thus if $\Gamma=\langle \alpha_1,\ldots,\alpha_d,\lambda_1,\ldots,\lambda_d\rangle$ then $A\seq\Sigma_d(\Gamma)$, and so $A$ is sufficiently sparse. The multiplicity assumption on the roots is necessary, as $a_n=2^n+n$ is given by the recurrence $a_n=4a_{n-1}-5a_{n-2}+2a_{n-3}$, with characteristic polynomial $(x-2)(x-1)^2$, and $A=\{a_n:n\in\N\}$ is not sufficiently sparse (see Remark \ref{rem:bad}).

In \cite{PoLa}, Lambotte and Point prove superstability for $\cZ(A)$, where $A$ is given by a linear recurrence relation satisfying certain properties. In particular, it covers the case that the characteristic polynomial is irreducible with a real root $\theta>1$ of absolute value strictly greater than the modulus of any other root. This generalizes the results of \cite{CoSS} applied to recurrence relations since, in order for such a recurrence to be geometrically sparse, one needs the modulus of all other roots to be on the closed complex unit disk. 
\end{remark}

\subsection{Induced structure on multiplicative monoids}

In this section, we consider the induced structure $\Gamma^{\cZ}$ where $\Gamma$ is a finitely generated multiplicative submonoid of $\Z^+$.

\begin{definition}
Fix $Q\seq\N_{\geq2}$.
\begin{enumerate}
\item Let $\Gamma(Q)$ be the multiplicative submonoid of $\N$ generated by $Q$. 
\item $Q$ is \textbf{multiplicatively independent} if, for all $k\geq 1$, $\bar{q}\in Q^k$, and $\bar{n}\in\Z^k$, $q_1^{n_1}\ldots q_k^{n_k}=1$ implies $n_i=0$ for all $1\leq i\leq k$ (equivalently, $\{\log q:q\in Q\}$ is $\Q$-linearly independent).
\end{enumerate}
\end{definition}

In the case that $Q$ is a set of primes, $\Gamma(Q)$ is sometimes called the set of \emph{$Q$-smooth numbers}. We first show that if $Q\seq\N_{\geq 2}$ is finite and multiplicatively independent, then $\Gamma(Q)^{\cZ}$ is a virtual reduct of a structure which is monadically stable of finite $U$-rank. We now define this structure.

\begin{definition}
Fix $d\geq 1$. Let $\cN_{d,\ms}=(\N^d,\ms_1,\ldots,\ms_d)$ where, for $1\leq i\leq d$ and $\nbar\in\N^d$, $\ms_i(n_1,\ldots,n_d)=(n_1,\ldots,n_{i-1},n_i+1,n_{i+1},\ldots,n_d)$.
\end{definition}

\begin{proposition}\label{prop:monstab}
For any $d\geq 1$, $\cN_{d,\ms}$ is monadically stable of $U$-rank $1$.
\end{proposition}
\begin{proof}
Fix $d\geq 1$, and let $\cZ_{d,s}=(\Z^d,\ms_1,\ldots,\ms_d)$. Note that $\ms_1,\ldots,\ms_d$ are commuting bijections on $\Z^d$. Let $T$ be the complete theory of $\cZ_{d,s}$, and let $T_1$ be the complete theory of the expansion of $\cZ_{d,\ms}$ by unary predicates for all subsets of $\Z^d$. We first show $T_1$ has $U$-rank $1$, and the claim is that this follows directly from \cite[Theorem 3.3]{LaskMAS}, which gives four equivalent conditions for a theory to be  ``mutually algebraic". Condition $(4)$  (of \cite[Theorem 3.3]{LaskMAS}) states that such theories are $U$-rank $1$ with trivial pregeometry, and so $T$ is mutually algebraic. Condition (2) then implies that any expansion of any model of $T$ by unary predicates is a mutually algebraic structure, and so, with condition (1), it follows that $T_1$ is mutually algebraic. By condition (4) again, $T_1$ has $U$-rank $1$, as claimed. Finally, since $\cN_{d,\ms}$ is the $\cZ_{d,\ms}$-induced structure on $\N^d\seq\Z^d$, we have the desired result.   
\end{proof}

\begin{remark}
The previous result can be shown ``by hand", by demonstrating quantifier elimination for the expansion of $\cN_{d,\ms}$ by all unary predicates and inverses for the successor functions, in a similar fashion as \cite[Proposition 5.9]{CoSS} (which deals with the case $d=1$). Then this expansion is superstable of $U$-rank $1$ by the characterization of such theories in \cite[Theorem 21]{BPW}, using the notion of \emph{quasi-strong-minimality}. Another interesting way to see just monadic stability of $\cN_{d,\ms}$ (without the bound on $U$-rank) is as follows. Consider the directed graph relation $R$ on $\N^d$ such that $R(\xbar,\ybar)$ holds if and only if $\ybar=\ms_i(\xbar)$ for some $1\leq i\leq d$. If $R^*(\xbar,\ybar)=R(\xbar,\ybar)\vee R(\ybar,\xbar)$, then $(\N^d,R^*)$ is a graph in which every vertex has degree at most $2d$. By \cite[Theorem 1.4]{IvSF}, $(\N^d,R)$ is monadically stable. Now it is straightforward to show that $\cN_{d,\ms}$ is a reduct of the expansion of $(\N^d,R)$ by unary predicates $P_i=\{\nbar\in\N^d:n_i\in 2\N\}$ for $1\leq i\leq d$. 
\end{remark}

The next result identifies the $\cZ$-induced structure on multiplicative submonoids of $\Z^+$ of the form $\Gamma(Q)$, where $Q\seq\N_{\geq 2}$ is finite and multiplicatively independent. 

\begin{theorem}\label{thm:FG0}
Suppose $Q=\{q_1,\ldots,q_d\}\seq\N_{\geq 2}$ is multiplicatively independent, and define $f:\N^d\to \Gamma(Q)$ such that $f(n_1,\ldots,n_d)=q_1^{n_1}\ldots q_d^{n_d}$.
\begin{enumerate}[$(a)$]
\item $\Gamma(Q)_0^{\cZ}$ is interdefinable with $\cN_{d,\ms}$ via $f$.
\item  Let $\cN^{\mod}_{d,\ms}$ be the expansion of $\cN_{d,\ms}$ by unary predicates for sets of the form
\[
\{\nbar\in\N^d:f(\nbar)\equiv_m r\},
\]
 for all $0\leq r<m<\infty$. Then $\Gamma(Q)^{\cZ}$ is interdefinable with $\cN^{\mod}_{d,\ms}$ via $f$. 
 \end{enumerate}
\end{theorem}
\begin{proof}
Note that part $(b)$ follows from part $(a)$ by Proposition \ref{prop:cong}. So it suffices to prove part $(a)$.  The map $f$ is clearly surjective, and the assumption that $Q$ is multiplicatively independent says 
precisely that $f$ is injective. We first show that $\Gamma(Q)_0^{\cZ}$ is a virtual reduct of $\cN_{d,\ms}$ (via $f\inv$). So fix $k\geq 1$, $r\in\Z$, and $\cbar\in\{\nv 1,1\}^k$, and define 
\[
X(\cbar,r)=\left\{(\nbar_1,\ldots,\nbar_k)\in\N^{dk}:\sum_{i=1}^k c_if(\nbar_i)\de r\right\}.
\]
We want to show $X(\cbar,r)$ is definable in $\cN_{d,\ms}$. If $r\neq 0$ then $X(\cbar,r)$ is finite by Fact \ref{fact:ESS}$(ii)$. So we may assume $r=0$. Let $X:=X(\cbar,0)$.  By Fact \ref{fact:ESS}$(i)$, there are $\kn(1),\ldots,\kn(t)\in X$ such that, for all $\kn\in X$ there is $1\leq u\leq t$ and $\lambda\in\R^*$ such that, for $1\leq i\leq k$, $f(\kn_i)=\lambda f(\kn(u)_i)$.  Any such $\lambda$ must be of the form $q_1^{m_1}\cdot\ldots\cdot q_d^{m_d}$ for some $\mbar\in\Z^d$.  Now, given $\kn=(\kn_1,\ldots,\kn_k)\in \N^{dk}$ and $\mbar\in\Z^d$, let 
\[
\kn\oplus \mbar=(\kn_1+\mbar,\ldots,\kn_k+\mbar)\mand \kn\ominus \mbar=(\kn_1-\mbar,\ldots,\kn_k-\mbar).
\] 
Note that if $\kn \in X$, $\mbar\in \Z^d$, and $\kn\oplus\mbar\in\N^{dk}$, then $\kn\oplus\mbar\in X$.  Moreover, any $\mbar\in\Z^d$ is of the form $\mbar_1-\mbar_2$ for some $\mbar_1,\mbar_2\in\N^d$, and for any $\kn\in\N^{dk}$ there are only finitely many $\mbar\in\N^d$ such that $\kn\ominus\mbar\in\N^{dk}$. Therefore, after enlarging the set $\{\kn(1),\ldots,\kn(t)\}$ if necessary, we may assume 
\[
X=\bigcup_{u=1}^t\{\kn(u)\oplus \mbar:\mbar\in\N^d\}.
\]
So it suffices to fix $\kn=(\kn_1,\ldots,\kn_k)\in\N^{dk}$  and show that the set $Z(\kn)=\{\kn\oplus \mbar:\mbar\in\N^d\}$ is definable in $\cN_{d,\ms}$. Given $\rbar\in\N^d$, define the function $\ms^{\rbar}=\ms_1^{r_1}\circ\ldots\circ\ms_d^{r_d}$, which is definable in $\cN_{d,\ms}$. Given $\mbar\in\N^d$, define 
\[
I(\mbar)=\bigcap_{i=1}^d\{\nbar\in\N^d:n_i\geq m_i\}=\bigcap_{i=1}^d\{\nbar\in\N^d:\exists \xbar~\ms_i^{m_i}(\xbar)= \nbar\},
\] 
and note that any such $I(\mbar)$ is definable in $\cN_{d,\ms}$. Given $1\leq i\leq k$ and $1\leq j\leq d$, let $r_{i,j}=\kn_{i,j}-\kn_{1,j}$. Let
\[
Z=\{(\ms^{\rbar_1}(\nbar),\ldots,\ms^{\rbar_k}(\nbar)):\nbar\in I(\kn_1)\},
\]
and note that $Z$ is definable in $\cN_{d,\ms}$. We claim $Z(\kn)=Z$. To see this, first fix $\mbar\in\N^d$. For $1\leq j\leq d$, set $n_j=m_j+\kn_{1,j}$, and note that $n_j\geq\kn_{1,j}$. By definition of $r_{i,j}$, we have $\kn\oplus\mbar=(\ms^{\rbar_1}(\nbar),\ldots,\ms^{\rbar_k}(\nbar))$. 
Conversely, fix $\nbar\in\N^d$ such that $n_j\geq\kn_{1,j}$ for all $1\leq j\leq d$. For $1\leq j\leq d$, let $m_j=n_j-\kn_{1,j}$. Then $\mbar\in\N^d$ and $\kn\oplus\mbar=(\ms^{\rbar_1}(\nbar),\ldots,\ms^{\rbar_k}(\nbar))$.

Finally, we must show $\cN_{d,\ms}$ is a virtual reduct of $\Gamma(Q)_0^{\cZ}$ (via $f$). For this, fix $1\leq i\leq d$, and note that
\[
f(\{(\mbar,\nbar)\in\N^d\times\N^d:\nbar=\ms_i(\mbar)\})=A^2\cap\{(x,y)\in\Z^2:y=q_ix\}\qedhere
\]
\end{proof}

\begin{remark}
When $d=1$, the notation and technical details in the previous proof become much simpler. In particular, this gives a short proof that for any integer $q\geq 2$, $(q^{\N})^{\cZ}$ is interdefinable with $\cN^{\mod}_{1,\ms}$ via $n\mapsto q^n$. This was first shown by Palac\'{i}n and Sklinos \cite{PaSk}. 
\end{remark}

\subsection{Main results and consequences}\label{sec:mggcors}

We can now prove the first main result.

\begin{proof}[Proof of Theorem \ref{thm:CMS}]
First, note that for any finitely generated multiplicative submonoid $\Gamma$ of $\Z^+$, there is some finite, multiplicatively independent $Q\seq\N_{\geq 2}$ such that $\Gamma\seq\Gamma(Q)$. So it suffices to assume $\Gamma$ has a multiplicatively independent finite generating set. The result then follows from Corollary \ref{cor:mainstab}, Corollary \ref{cor:mgss}, Proposition \ref{prop:monstab}, and Theorem \ref{thm:FG0}$(a)$. 
\end{proof}

Theorem \ref{thm:CMS} can be used to exhibit some new and interesting behavior in stable expansions of $(\Z,+)$. Recall that a subset $A\seq\Z^+$ is \emph{lacunary} if $\limsup_{n\to\infty}\frac{a_{n+1}}{a_n}>1$, where $(a_n)_{n=0}^\infty$ is a strictly increasing enumeration of $A$. In particular, $A\seq\Z^+$ is not lacunary if and only if $\lim_{n\to\infty}\frac{a_{n+1}}{a_n}=1$. The results of \cite{CoSS} and \cite{PoLa} provide examples of lacunary sets $A\seq\Z^+$ such that $\cZ(A)$ is stable and $\lim_{n\to\infty}\frac{a_{n+1}}{a_n}$ can be chosen arbitrarily close to $1$ (e.g. $A=\{\llbracket b^n\rrbracket :n\in\N\}$ for some fixed $b\in\R^{>1}$, by the main result of \cite{CoSS}, in which case both expansions define the ordering).  On the other hand, many well-known examples of unstable expansions of the form $\cZ(A)$ are given by non-lacunary sets. For example, if $A=\{n^k:n\in\N\}$ for some fixed $k\geq 1$, then $\cZ(A)$ defines the ordering since any sufficiently large integer is a uniformly bounded sum of elements of $A$ (see also \cite[Section 8.1]{CoSS}). In fact, if $k\geq 2$ then $\cZ(A)$ defines multiplication (see \cite[Proposition 6]{Bes}). Another example is the the expansion of $\cZ$ by the (non-lacunary) set  of primes, which is  defines the ordering  (whether multiplication is definable is an open question, see \cite{Bes}).  This motivates the question, asked in both \cite{CoSS} and \cite{PoLa}, of whether there is a stable expansion $\cZ(A)$ where $A\seq \Z^+$ is not lacunary. Theorem \ref{thm:CMS} provides such examples, when combined with the following characterization lacunary multiplicative submonoids of $\Z^+$.

\begin{fact}\label{fact:lacunary}
Let $\Gamma$ be a multiplicative submonoid of $\Z^+$, and suppose $Q\seq\N_{\geq 2}$ is a set of generators. The following are equivalent:
\begin{enumerate}[$(i)$]
\item $\Gamma$ is lacunary.
\item For any $q,r\in Q$, $\log_q r$ is rational.
\item There is an integer $q\geq 2$ such that $\Gamma\seq q^{\N}$.
\end{enumerate}
\end{fact}
\begin{proof}
The equivalence of $(ii)$ and $(iii)$ is straightforward, and $(iii)\Rightarrow (i)$ is obvious. The implication $(i)\Rightarrow (iii)$ is a result of Furstenberg \cite{FurstLac}. Specifically, in Part IV of \cite{FurstLac}, a submonoid $\Gamma$ is called \emph{non-lacunary} if it satisfies condition $(iii)$, and Lemma IV.1 of \cite{FurstLac} shows that if $\Gamma$ is non-lacunary and $(a_n)_{n=0}^\infty$ is an increasing enumeration of $\Gamma$, then $\lim_{n\to\infty}\frac{a_{n+1}}{a_n}=1$.
\end{proof}

So, by Theorem \ref{thm:CMS}, we have:

\begin{corollary}\label{cor:lacunary}
There exist infinite non-lacunary sets $A\seq\Z^+$ such that $\cZ(A)$ is superstable of $U$-rank $\omega$.
\end{corollary}

As remarked above, many  unstable expansions of the form $\cZ(A)$ are by non-lacunary subsets $A\seq\Z^+$ (e.g. primes or perfect powers). Moreover, it is not hard to find a lacunary set $A\seq\Z^+$ such that $\cZ(A)$ is unstable (e.g, $A=\{2^n+k:n,k\in\N,~k\leq n\}$; see \cite[Proposition 8.10]{CoSS}). On the other hand, there seems to be no previously known example of an unstable expansion $\cZ(A)$ where $A$ is \emph{strongly lacunary}, i.e., $\liminf_{n\to\infty}\frac{a_{n+1}}{a_n}>1$, where $(a_n)_{n=0}^\infty$ is a strictly increasing enumeration of $A$. In Theorem \ref{thm:badex}, we provide such examples, which, together with Corollary \ref{cor:lacunary}, give positive answers to both parts of Question 8.15 from \cite{CoSS}.

For the next consequence of Theorem \ref{thm:CMS}, we need a technical lemma.

\begin{lemma}\label{lem:unary}
For any integers $q_1,\ldots,q_d\geq 2$ there is a finitely generated multiplicative submonoid $\Gamma$ of $\Z^+$, and a subset $A\seq\Gamma$, such that $\cZ(q_1^{\N},\ldots, q_d^{\N})$ is a reduct of $\cZ(A)$.
\end{lemma}
\begin{proof}
Define the equivalence relation $\sim$ on $\N_{\geq 2}$ by $q\sim r$ if and only if $\log_q r$ is rational. Let $R_1,\ldots,R_t$ be the partition of $\{q_1,\ldots,q_d\}$ induced by $\sim$. By Fact \ref{fact:lacunary},  each $R_i$ generates a lacunary multiplicative submonoid of $\Z^+$, and so there is an integer $b_i\geq 2$ such that $R_i\seq b_i^{\N}$. Given $1\leq i\leq t$, let $u_i$ be the least common multiple of the set of \emph{positive integers} $\{\log_{b_i}r:r\in R_i\}$, and set $c_i=b_i^{u_i}$. By construction $c_i\not\sim c_j$ for all $1\leq i<j\leq t$.  Let $A=\bigcup_{i=1}^t c_i^{\N}\seq\Gamma(c_1,\ldots,c_t)$. 

We first show that, for all $1\leq i\leq t$, $c_i^{\N}$ is definable in $\cZ(A)$. So fix $1\leq i\leq t$, and let  $c=c_i$. We claim that $x\in c^{\N}$ if and only if $c^mx\in A$ for all $1\leq m\leq t$. The forward direction is clear, so fix $x\in\Z$ such that $c^mx\in A$ for all $1\leq m\leq t$. For $1\leq m\leq t$, we may fix $1\leq i_m\leq t$ and $n_m\geq 1$ such that $c^mx=c_{i_m}^{n_m}$. If $i_m=i$ for some $1\leq m\leq t$, then we have $x=c^{n_m-m}$ and so $x\in c^{\N}$, as desired. Otherwise, by pigeonhole there are $1\leq m<m_*\leq t$ and $1\leq j\leq t$ such that $j\neq i$ and $i_{m}=j=i_{m_*}$. Then $c^{m}x=c_j^{n_m}$ and $c^{m_*}x=c_j^{n_{m_*}}$, which implies $c^{m_*-m}=c_j^{n_{m_*}-n_m}$, contradicting $c\not\sim c_j$.

Finally, to finish the proof, it suffices to fix $1\leq i\leq d$ and show $q_i^{\N}$ is definable in $\cZ(c_j^{\N})$, where $1\leq j\leq t$ is such that $q_i\in R_j$. Let $q=q_i$, $b=b_j$, $u=u_j$, and $c=c_j=b^u$. By choice of $u$, there is some integer $v\geq 1$ such that $q=b^v$ and $v$ divides $u$. Fix $k\geq 1$ such that $u=kv$. We show that $x\in q^{\N}$ if and only if $q^mx\in c^{\N}$ for some $0\leq m\leq k-1$. For one direction, fix $x\in\Z$ such that $q^mx\in c^{\N}$ for some $0\leq m\leq k-1$. Then, for some $n\geq 1$, $q^mx=c^n=b^{nu}=b^{nkv}=q^{nk}$, and so $x=q^{nk-m}\in q^{\N}$. Conversely, suppose $x=q^n$ for some $n\geq 1$. Let $n=ak+r$ where $0<r\leq k$. Set $m=k-r$, and note $0\leq m\leq k-1$. Then
\[
q^mx=q^{k-r}q^{ak+r}=q^{(1+a)k}=c^{1+a}\in c^{\N},
\]
as desired.
\end{proof}

We can now state and prove the second main result of this section.

\begin{theorem}\label{thm:MUP}
For any integers $q_1,\ldots,q_d\geq 2$, $\cZ(q_1^{\N},\ldots,q_d^{\N})$ is superstable of $U$-rank $\omega$. Therefore  $\cZ(q^{\N})_{q\geq 2}$ is stable.
\end{theorem}
\begin{proof}
The first statement follows from Fact \ref{fact:PSexp}, Theorem \ref{thm:CMS}, and Lemma \ref{lem:unary}. For the second statement, recall that stability of a structure is equivalent to stability of all reducts to finite sublanguages.
\end{proof}

In particular, Theorem \ref{thm:MUP} gives the first known examples of proper stable expansions of $(\Z,+)$ by at least two unary predicates each of which is undefinable from the others. Using similar kinds of techniques one can construct a variety of interesting stable expansions of $(\Z,+)$. Here is one more example.

\begin{corollary}
Given $n\geq 1$, let $\Gamma_n$ be the multiplicative submonoid of $\Z^+$ generated by the first $n$ prime numbers. Then $\cZ(\Gamma_n)_{n\geq 1}$ is stable.
\end{corollary}
\begin{proof}
For any fixed $n\geq 1$, $\Gamma_k$ is definable in $\cZ(\Gamma_n)$ for all $k\leq n$. Now apply Theorem \ref{thm:CMS}.
\end{proof}

We still do not have an example of a strictly stable expansion of $(\Z,+)$ or a superstable expansion with $U$-rank greater than $\omega$ (recall that finite $U$-rank is prohibited by \cite{PaSk}). This motivates the following questions.

\begin{question}
Are $\cZ(q^{\N})_{q\geq 2}$ and $\cZ(\Gamma_n)_{n\geq 1}$ superstable? If so, what are their $U$-ranks? Is the expansion of $\cZ$ by all finitely generated multiplicative submonoids of $\Z^+$ stable?
\end{question}

It is worth mentioning  that the structure $(\Z,+,<,q^{\N})$, for some fixed $q\geq 2$, is also known to be model theoretically tame (e.g. NIP) by quantifier elimination results of Cherlin and Point \cite{ChPo} (see also the remarks after Theorem \ref{thm:badex}). On the other hand, given multiplicatively independent $p$ and $q$, whether $(\Z,+,<,p^\N,q^\N)$ is model theoretically tame appears to be open. In this structure, one can express statements about the number of solutions to equations such as $p^x-q^y=z$, which is a well-studied topic related to rational approximation of algebraic numbers  (see, e.g, \cite{StTij}).  Incidentally, by a result of B\`{e}s \cite{Bes2}, $(\Z,+,V_p(x),q^{\N})$ defines multiplication, where $V_p(x)$ is the largest power of $p$ dividing $x$.

\section{Independent geometric sequences}

The results in this section are motivated by the main result of \cite{CoSS} showing that $\cZ(A)$ is stable for any \emph{geometrically sparse} subset of $\N$. The following is the main definition.

\begin{definition}\cite{CoSS}
A strictly increasing sequence $(\lambda_n)_{n=0}^\infty$ of positive real numbers is \textbf{geometric} if the set $\{\frac{\lambda_n}{\lambda_m}:0\leq m\leq n\}$ is closed and discrete.
\end{definition}

The following crucial fact about geometric sequences is proved in \cite[Lemma 7.3]{CoSS} by modifying an unpublished argument by Poonen (which is similar in flavor to Fact \ref{fact:ESS}).

\begin{fact}\label{fact:Poon}
Suppose $S\seq\R^+$ is such that $\{\frac{s}{t}:s,t\in S,~t\leq s\}$ is closed and discrete. For any $k\geq 1$ there is some $\epsilon>0$ such that, for any $\cbar\in\{\nv 1,1\}^k$ and $\lambda_1,\ldots,\lambda_k\in S$, if $\sum_{i\in I}c_i\lambda_i\neq 0$ for all nonempty $I\seq[k]$, then 
\[
|c_1\lambda_1+\ldots+c_k\lambda_k|\geq \max\{\epsilon\lambda_1,\ldots,\epsilon\lambda_k\}.
\]
\end{fact}

Throughout this section, all enumerated infinite sequences of integers or reals are assumed to be strictly increasing. Given functions $f,g:\N\to\R^+$, we say $f(n)$ is $O(g(n))$ if there is some constant $C$ such that $f(n)\leq Cg(n)$ for all $n\in\N$, and we say $f(n)$ is $o(g(n))$ if, given $\epsilon>0$, $f(n)\leq \epsilon g(n)$ for sufficiently large $n$. 

\begin{definition}\label{def:gs}\cite{CoSS}
A set $A\seq\Z$ is \textbf{geometrically sparse} if there is a set $B=(b_n)_{n=0}^\infty\seq\N$ and a geometric sequence $(\lambda_n)_{n=0}^\infty$ such that $|b_n-\lambda_n|$ is $O(1)$ and $A\seq B+F$ for some finite $F\seq\Z$.
\end{definition}

The following is the main result in \cite{CoSS}.

\begin{fact}\label{fact:CoSS}
If $A\seq\Z$ is geometrically sparse then $\cZ(A)$ is superstable of $U$-rank $\omega$.
\end{fact}

The natural question is whether the $O(1)$ restriction in Definition \ref{def:gs} can be relaxed. For instance, many examples in \cite{PoLa} of sets $A\seq\N$, for which $\cZ(A)$ is stable, are not geometrically sparse, but are such that $|a_n-\lambda_n|$ is $o(\lambda_n)$  for some geometric sequence $(\lambda_n)_{n=0}^\infty$. This motivates the following example, which is also relevant to the discussion of strongly lacunary subsets of $\Z^+$ after Corollary \ref{cor:lacunary}. 

\begin{definition}
Given $q\geq 2$, let $A_q=\{q^n+n:n\in\N\}$.
\end{definition}

\begin{remark}\label{rem:bad}
For any $q\geq 2$, $A_q$ is strongly lacunary since, if $a_n=q^n+n$, then $\lim_{n\to\infty}\frac{a_{n+1}}{a_n}=q$. Moreover, the sequence $(q^n)_{n=0}^\infty$ is geometric and $|a_n-q^n|$ is $o(q^n)$. It is also worth pointing out that $A_q$ is not sufficiently sparse since $(q-1)n=qa_{n+1}-a_{n+2}-qa_1+a_2$ for any $n\in\N$, and so $\Sigma_{2q+2}(\pm A_q)=(q-1)\Z$.
\end{remark}

The first main result in this section is that $\cZ(A_q)$ is unstable for any $q\geq 2$. This question, even for just $A_2$, is discussed in both \cite{CoSS} and \cite{PoLa}. In particular, $A_2$ often emerges as an obstacle in attempts to generalize the ``sparsity assumptions" on a set $A\seq\Z^+$ which ensure $\cZ(A)$ is stable. For this reason, settling the question of whether $\cZ(A_2)$ is stable provides a concrete and useful barrier in future work toward generalizing the methods in \cite{CoSS}, \cite{PoLa}, and \cite{PaSk}. For the proof, we first recall asymptotic density of subsets of $\Z^+$.

\begin{definition}
Given $X\seq\Z^+$, define
\[
d^*(X)=\limsup_{n\to\infty}\frac{|X\cap [n]|}{n}\mand d_*(X)=\liminf_{n\to\infty}\frac{|X\cap[n]|}{n},
\]
called the \textbf{upper asymptotic density} and \textbf{lower asymptotic density} of $X$, respectively.
\end{definition}

\begin{theorem}\label{thm:badex}
For any $q\geq 2$, $\cZ(A_q)$ is interdefinable with $(\Z,+,<,q^x)$, where $q^x$ denotes the function $x\mapsto q^x$ with domain $\N$. In particular, $\cZ(A_q)$ is unstable.
\end{theorem}
\begin{proof}
Note that $A_q$ is definable in $(\Z,+,<,q^x)$, since $\N$ is definable. So it suffices to show $<$ and $q^x$ are definable in $\cZ(A_q)$. Given $n\in\N$, let $a_n=q^n+n$. We first assume $<$ is definable in $\cZ(A_q)$, and use this to show $q^x$ is definable.  If $<$ is definable in $\cZ(A_q)$ then the successor function $s: A_q\to A_q$, such that $s(a_n)=a_{n+1}$, is definable. Since $(q-1)q^{\N}=\{s(a)-a-1:a\in A_q\}$, we have that $q^{\N}$ is definable in $\cZ(A_q)$. Now $y=q^x$ if and only if $y\in q^{\N}$, $0\leq x<y$, and $x+y\in A_q$.

Now we show that $<$ is definable in $\cZ(A_q)$. Since $\cZ(A_q)$ expands the group structure, it suffices to show $\N$ is definable. As observed in Remark \ref{rem:bad}, the following identity holds for any $n\in\N$,
\[
(1-q)n=a_{n+2}-qa_{n+1}+q-2.
\]
Define
\[
 B=(q-1)\Z\cap\{u-qv+q-2:u,v\in A_q\},
 \]
Then $B$ is definable in $\cZ(A_q)$, and $\{b\in B:b\leq 0\}=(q-1)\Z_{\leq 0}$. Let $X=\{b\in B:b>0\}$ and $C=((q-1)\Z^+)\backslash X$. Then $C$ is definable in $\cZ(A_q)$ since $C=((q-1)\Z)\backslash B$. 

To prove $\cZ(A_q)$ defines $\N$, we will show $d^*(X)=0$. This will suffice since, if $d^*(X)=0$, then $d_*(C)=d_*((q-1)\Z^+)-d^*(X)=\frac{1}{q-1}>0$ and so, by a result of Nash and Nathanson \cite[Lemma 1]{NaNa}, there is some $n>0$ such that $\Sigma_n(C\cup\{0,1\})=\N$. 

To show $d^*(X)=0$, let $V$ be the set of $z\in\Z^+$ such that $z=u-qv+q-2$ for some $u,v\in A_q$. Then $X=V\cap(q-1)\Z^+$, so it suffices to show $d^*(V)=0$. Any element of $V$ is of the form 
\[
a_l-qa_k+q-2=q^l-q^{k+1}+l-kq+q-2,
\]
for some $k,l\geq 0$. Since every element of $V$ is positive, we claim that $k$ and $l$ must also satisfy $k+1\leq l$. Indeed, if $l\leq k$, then
\[
a_l-qa_k+q-2\leq q^l(1-q^{k-l+1})+k+q-kq-2\leq 0.
\]
So, setting $r=l-k-1$, we have that every element of $V$ is of the form,
\[
f(k,r):=q^{k+1}(q^r-1)+k+r-kq+q-1,
\]
for some $k\geq 0$ and $r\geq 0$. Given $n\geq 1$, define
\[
g(n)=|\{(k,r)\in\N\times\N:1\leq f(k,r)\leq n\}|.
\]
Then $|V\cap [n]|\leq g(n)$ and so, to show $d^*(V)=0$, it suffices to show $g(n)$ is $o(n)$. In the following, $\log$ denotes $\log_q$. First note that if $r\geq 1$ then
\[
f(k,r)\geq q^{k+1}-kq\geq q^k,
\]
and so if $f(k,r)\leq n$ then $k\leq \log n$. On the other hand, $f(k,0)=(k-1)(1-q)$, and so $f(k,0)\leq 0$ whenever $k\geq 1$. Thus, for any $k,r\geq 0$, if $1\leq f(k,r)\leq n$ then $k\leq \log n$. Now we also have
\[
f(k,r)\geq q^r-1-kq+q-1\geq q^r-kq.
\]
Therefore, for any $k,r\geq 0$, if $1\leq f(k,r)\leq n$ then
\[
q^r\leq n+kq\leq n+q\log n\leq (q+1)n,
\]
 and so $r\leq c+\log n$ for $c=\log(q+1)$. Altogether, for any $k,r\geq 0$, if $1\leq f(k,r)\leq n$ then $k\leq \log n$ and $r\leq c+\log n$, and so
\[
g(n)\leq (c+\log n)\log n,
\]
So $g(n)$ is $o(n)$, as desired.
\end{proof}

Quantifier elimination and decidability for $(\Z,+,<,q^x)$ was investigated by Semenov \cite{Semenov}. In \cite{ChPo}, Cherlin and Point give a detailed account for $q=2$. It is also interesting to note that $(\Z,+,<,q^{\N})$ is  a \emph{proper} reduct of $(\Z,+,<,q^x)$. In particular, if $V_q\colon \Z^+\to q^{\N}$ is such that $V_q(x)$ is the greatest power of $q$ dividing $x$, then $(\Z,+,<,V_q(x))$ is decidable while $(\Z,+,<,q^x,V_q(x))$ is undecidable.\footnote{The history of the first claim is given in \cite{MiVi}; the second claim is shown in \cite{ChPo} for $q=2$, and the generalization to arbitrary $q\geq 2$ is straightforward.} It follows that $q^x$ is not definable in $(\Z,+,<,V_q(x))$, and thus not definable in $(\Z,+,<,q^{\N})$.

Theorem \ref{thm:badex} shows that if $O(1)$ is weakened to $O(n)$ in Definition \ref{def:gs}, then the resulting analog of Fact \ref{fact:CoSS} fails. So we ask:

\begin{question}
Is there a set $A\seq\N$ such that $\cZ(A)$ is unstable and $|a_n-\lambda_n|$ is $o(n)$, for some geometric sequence $(\lambda_n)_{n=0}^\infty$?
\end{question} 

The second main result of this section is that the asymptotic bound $O(1)$ underlying Fact \ref{fact:CoSS} can be substantially weakened if one makes further assumptions on the associated geometric sequence.

\begin{definition}\label{def:IS}
A set $A\seq\Z$ is \textbf{independently sparse} if there is a set $B=(b_n)_{n=0}^\infty\seq\N$ and a geometric sequence $(\lambda_n)_{n=0}^\infty$ such that $\{\lambda_n:n\in\N\}$ is $\Q$-linearly independent, $|b_n-\lambda_n|$ is $o(\lambda_n)$, and $A\seq B+F$ for some finite $F\seq\Z$.
\end{definition}

\begin{example}\label{ex:IS}
Suppose $\tau>1$ is transcendental. Then the sequence $(\tau^n)_{n=0}^\infty$ is $\Q$-linearly independent and geometric, and so a straightforward example of an independently sparse set, which is not geometrically sparse, is the set enumerated by $a_n=\llbracket\tau^n\rrbracket+n$ (where $\llbracket\cdot \rrbracket$ denotes integer part). More generally, fix algebraically independent reals $\tau_1,\ldots,\tau_k>1$ and define $(\lambda_n)_{n=0}^\infty$ such that $\lambda_0=\tau_1$ and $\lambda_{n+1}=c_n\tau_{i_n}\lambda_n$ for arbitrarily chosen $c_n\in\Z^+$ and $i_n\in[k]$. Then $(\lambda_n)_{n=0}^\infty$ is $\Q$-linearly independent and geometric. So for any function $g\colon\N\to\R^+$ such that $g(n)$ is $o(\lambda_n)$ (e.g. if $g(n)$ is $2^{o(n)}$), the set $A=\{\llbracket\lambda_n+g(n)\rrbracket:n\in\N\}$ is independently sparse (taking $B=A$ and $F=\{0\}$ in Definition \ref{def:IS}). 
\end{example}

\begin{remark}
The class of independently sparse sets satisfies the following strong closure property: if $A\seq\Z$ is independently sparse and $F\seq\Z$ is finite, then any subset of $A+F$ is independently sparse. The class of geometrically sparse sets also satisfies this property.
\end{remark}

The main result of this section is the following theorem.

\begin{theorem}\label{thm:IS}
If $A\seq\Z$ is independently sparse and infinite then $\cZ(A)$ is superstable of $U$-rank $\omega$.
\end{theorem}

\begin{remark}
Theorem \ref{thm:IS} generalizes the ``transcendental limit" case of a result of Lambotte and Point \cite[Theorem 3.8]{PoLa}, which shows that if $A=(a_n)_{n=0}^\infty$ is such that $|a_n-\alpha\tau^n|$ is $o(\tau^n)$ for some $\alpha>0$ and transcendental $\tau>1$, then $\cZ(A)$ is superstable of $U$-rank $\omega$.
\end{remark}

The proof strategy for Theorem \ref{thm:IS} is the same as in Section \ref{sec:MGG}. We first show that independently sparse sets are sufficiently sparse, and then we interpret the induced structure on such a set in a superstable structure of finite $U$-rank.

 \begin{proposition}\label{prop:ISss}
If $A\seq\N$ is independently sparse then it is sufficiently sparse.
\end{proposition}
\begin{proof}
First, for any $A\seq\Z$, $A$ is sufficiently sparse if and only if, for all $n\geq 1$, $d_*(\Sigma_n(\pm A)\cap \Z^+)=0$ (this is a direct consequence of \cite{NaNa}; see \cite[Proposition 4.2]{CoSS}). It follows that if $A\seq\Z$ is sufficiently sparse and $F\seq\Z$ is finite, then $A+F$ is sufficiently sparse (see, e.g, \cite[Lemma 4.3]{CoSS}). Altogether, to prove the claim, it suffices to fix $A=(a_n)_{n=0}^\infty\seq\N$ and a $\Q$-linearly independent geometric sequence $(\lambda_n)_{n=0}^\infty$, such that $|a_n-\lambda_n|$ is $o(\lambda_n)$, and prove that $A$ is sufficiently sparse.

Given $k\geq 1$ and $\cbar\in\{\nv 1,1\}^k$, let $X(\cbar)=\{\nbar\in\N^k:c_1a_{n_1}+\ldots+c_ka_{n_k}\neq 0\}$, and then define $X_*(\cbar)=\{\nbar\in\N^k:(n_i)_{i\in I}\in X((c_i)_{i\in I})\text{ for all nonempty $I\seq[k]$}\}$.

\noit{Claim}: For any $k\geq 1$, there is $\epsilon>0$ and $N>0$ such that, for any $\cbar\in\{\nv 1,1\}^k$ and $\nbar\in X_*(\cbar)$, if $\max\nbar\geq N$ then $|c_1a_{n_1}+\ldots+c_ka_{n_k}|\geq\epsilon\lambda_{\max\nbar}$. 

\noit{Proof}: Given $k\geq 1$, $\cbar\in\{\nv 1,1\}^k$, and $\nbar\in\N^k$, set $g(\cbar,\nbar)=c_1\lambda_{n_1}+\ldots+c_k\lambda_{n_k}$. Fix $k\geq 1$, $\cbar\in\{\nv 1,1\}^k$, and $\nbar\in\N^k$, and suppose $g(\cbar,\nbar)=0$. We first show $\nbar\not\in X(\cbar)$. Let $\cP$ be the partition of $[k]$ induced by the equivalence relation $n_i=n_j$. Given $P\in \cP$, let $n_P$ be the unique value of $n_i$ for $i\in P$, and let $c_P=\sum_{i\in P}c_i$. Then $n_P\neq n_Q$ for distinct $P,Q\in\cP$, and  $\sum_{P\in \cP}c_P\lambda_{n_P}=g(\cbar,\nbar)=0$. By $\Q$-linear independence, $c_P=0$ for all $P\in\cP$. Therefore $c_1a_{n_1}+\ldots+c_ka_{n_k}=\sum_{P\in\cP}c_Pa_{n_P}=0$. 

By Fact \ref{fact:Poon} and the above, there is $\epsilon>0$ such that $|g(\cbar,\nbar)|\geq 4\epsilon\lambda_{\max\nbar}$ for all $\cbar\in\{\nv 1,1\}^k$ and $\nbar\in X_*(\cbar)$. For $n\in\N$, let $\theta_n=a_n-\lambda_n$.  Since $|a_n-\lambda_n|$ is $o(\lambda_n)$, there is $M>0$ such that for all $n\geq M$, $|\theta_n|\leq \frac{2\epsilon}{k}\lambda_n$. Define
\[
\theta=\max\{|c_1\theta_{n_1}+\ldots+c_l\theta_{n_l}|:l< k,~\cbar\in\{\nv 1,1\}^l,~n_i<M\text{ for all $1\leq i\leq l$}\}.
\]
For any $\cbar\in\{\nv 1,1\}^k$ and $\nbar\in\N^k$, if $\max\nbar\geq M$ then
\[
|c_1\theta_{n_1}+\ldots+c_k\theta_{n_k}|\leq \theta+2\epsilon\lambda_{\max\nbar}.
\]
Therefore, for any $\cbar\in\{\nv 1,1\}^k$ and $\nbar\in X(\cbar)$, if $\max\nbar\geq M$ then
\[
|c_1a_{n_1}+\ldots+c_ka_{n_k}|\geq |g(\cbar,\nbar)|-|c_1\theta_{n_1}+\ldots+c_k\theta_{n_k}|\geq 2\epsilon\lambda_{\max \nbar}-\theta.
\]
Now choose $N\geq M$ such that $\epsilon\lambda_N\geq \theta$. \claim

Fix $k\geq 1$. We want to show $\Sigma_k(\pm A)$ does not contain a nontrivial subgroup of $\Z$. Let $B=\Sigma_k(\pm A)\cap\Z^+$. We will show $d^*(B)=0$.  Define $f\colon B\to\N$ so that, given $x\in B$, $f(x)=\max\nbar$ where $\nbar\in \N^l$ for some $l\leq k$ and $c_1a_{n_1}+\ldots+c_la_{n_l}\de x$ for some $\cbar\in\{\nv 1,1\}^l$. By the claim, we may fix $\epsilon>0$ and $N>0$ such that for all $x\in B$, if $f(x)\geq N$ then $x\geq \epsilon\lambda_{f(x)}$. Let $b=\inf\{\frac{\lambda_{n+1}}{\lambda_n}:n\in\N\}$. Then $b>1$ since $(\lambda_n)_{n=0}^\infty$ is geometric. For any $n\geq N$, if $x\in B\cap[n]$ then $f(x)\leq \max\{N,\log_b\frac{n}{\epsilon\lambda_0}\}\leq c\log n$ for some $c=c(k)$. Fix $n\in\N$. For $m\in \N$, let $B_m=\{x\in B\cap[n]:f(x)=m\}$. We have shown that $B\cap [n]=\bigcup_{m\leq c\log n}B_m$. By definition of $f$, $|B_m|\leq (3m)^k$ for any $m\in\N$. Altogether 
\[
|B\cap [n]|\leq\sum_{m\leq c\log n}(3m)^k\leq (c\log n)(3c\log n)^k=3^kc^{k+1}(\log n)^{k+1}.
\]
Altogether, we have  $d^*(B)=0$, as desired.
\end{proof}

For the rest of the section, we fix $A\seq\Z$, which is independently sparse witnessed by $B=(b_n)_{n=0}^\infty\seq\N$, $(\lambda_n)_{n=0}^\infty\seq\R^+$, and $F\seq\Z$. We assume $A=(a_n)_{n=0}^\infty$ is infinite, and thus we can construct a function $f\colon \N\to\N$ such that, for each $n\in\N$, $a_n=b_{f(n)}+r_n$ for some $r_n\in F$. By enlarging $F$, we may assume without loss of generality that $f$ is \emph{weakly increasing}, i.e. $m\leq n$ implies $f(m)\leq f(n)$. Given $k\geq 1$ and $\nbar\in\N^k$, let $\cP(\nbar)$ be the partition of $[k]$ induced by the equivalence relation $f(n_i)=f(n_j)$. 

\begin{definition}
Given $k\geq 1$, $\cbar\in\{\nv 1,1\}^k$, and $r\in\Z$, define
\[
A(\cbar,r)=\{\nbar\in\N^k:c_1a_{n_1}+\ldots+c_ka_{n_k}=r\}.
\]
Define $A_0(\cbar,r)=\left\{\nbar\in A(\cbar,r):\sum_{i\in P}c_i\neq 0\text{ for all }P\in\cP(\nbar)\right\}$.
\end{definition}

Let $\cN^1_{\ms}$ denote the expansion of $\cN_{1,\ms}=(\N,x\mapsto x+1)$ by unary predicates for all subsets of $\N$. Our goal is to interpret $A^{\cZ}_0$ in $\cN^1_{\ms}$, and the next lemma shows that, using the successor function and arbitrary unary predicates, we may reduce to sets of the form $A_0(\cbar,r)$. The proof is elementary but technical. 

\begin{lemma}\label{lem:IStech}
Suppose $A_0(\cbar,r)$ is definable in $\cN^1_{\ms}$ for all $k\geq 1$, $\cbar\in\{\nv 1,1\}^k$, and $r\in\Z$. Then $A(\cbar,r)$ is definable in $\cN^1_{\ms}$ for all $k\geq 1$, $\cbar\in\{\nv 1,1\}^k$, and $r\in\Z$.
\end{lemma}
\begin{proof}
Assume $A_0(\cbar,r)$ is definable in $\cN^1_{\ms}$ for all $k\geq 1$, $\cbar\in\{\nv 1,1\}^k$, and $r\in\Z$. We may fix an integer $K\geq 0$ such that, for all $m,n\in\N$, if $f(m)=f(n)$ then $|m-n|\leq K$. We use $[\nv K,K]$ for the interval of integers from $\nv K$ to $K$. Fix $k\geq 1$, $\cbar\in\{\nv 1,1\}^k$, and $r\in\Z$.

Let $\P$ be the set of partitions of $[k]$. Fix $\cP\in\P$. Then $\cP$ determines the following objects (which depend only on initial choice of $\cP$). Given $P\in\cP$, let $c_P=\sum_{i\in P}c_i$. Let $\cQ=\{P\in\cP:c_P=0\}$, $E=\prod_{P\in\cQ}[\nv K,K]^P$, and $I=\bigcup \cQ$. Let $\Sigma(\cP)=F^I$. 

For a fixed $\cP\in\P$ and $\sigma=(s_i)_{i\in I}\in\Sigma(\cP)$, we define a set $X(\cP,\sigma)$ as follows. Identify elements of $E$ as $\ku=(\ubar_P)_{P\in\cQ}$, where $\ubar_P\in [\nv K,K]^P$ for all $P\in\cQ$. Identify elements of $\N^{\cQ}$ as $\kn=(n_P)_{P\in\cQ}$.  Given $\ku\in E$ and $\kn\in\N^{\cQ}$, define $\ku\oplus\kn$ to be the tuple $\nbar\in\N^I$ such that for $P\in\cQ$ and $i\in P$, $n_i=n_P+u_{P,i}$. Given $P\in\cQ$ and $\ubar\in[\nv K,K]^P$, define
\[
U(P,\ubar)=\left\{n\in\N:\text{\begin{tabular}{l} $f(n+u_i)=f(n+u_i)$ for all $i,j\in P$, and\\ $a_{n+u_i}=b_{f(n+u_i)}+s_i$ for all $i\in P$.\end{tabular}}\right\}
\]
Given $P\in\cQ$, set $s_P=\sum_{i\in P}c_is_i$. Let $s=r-\sum_{P\in\cQ}s_P$. Note that, by construction, we may view $A_0((c_i)_{i\not\in I},s)$ as a subset of $\N^{[k]\backslash I}$. Given $\mbar\in \N^I$ and $\nbar\in\N^{[k]\backslash I}$, let $\nbar\otimes\mbar$ be the tuple $\vbar\in\N^k$ such that $v_i=m_i$ if $i\in I$ and $v_i=n_i$ if $i\not\in I$. Finally, we define
\[
X(\cP,\sigma)=\bigcup_{\ku\in E}\left\{\nbar_*\otimes (\ku\oplus\kn):\nbar_*\in A_0((c_i)_{i\not\in I},s),~\kn\in\prod_{P\in\cQ}U(P,\ubar_P)\right\}.
\]
Given a fixed $\ku\in E$, the set $\{\ku\oplus\kn:\kn\in\prod_{P\in \cQ}U(P,\ubar_P)\}$ is definable in $\cN^1_{\ms}$ using the successor function and unary predicates. By assumption, $A_0((c_i)_{i\not\in I},s)$ is definable in $\cN^1_{\ms}$. Therefore, since $E$ is finite, $X(\cP,\sigma)$ is definable in $\cN^1_{\ms}$. 

Let $\Delta=\{(\cP,\sigma):\cP\in\P,~\sigma\in \Sigma(\cP)\}$, and note that $\Delta$ is finite. To finish the proof, we show
\[
A(\cbar,r)=\bigcup_{(\cP,\sigma)\in\Delta}X(\cP,\sigma).
\]
For the right-to-left direction, fix $(\cP,\sigma)\in\Delta$, $\ku\in E$, $\nbar_*\in A_0((c_i)_{i\not\in I},s)$, and $\kn\in\prod_{P\in\cQ}U(P,\ubar_P)$. We show $\nbar_*\otimes(\ku\oplus\kn)\in A(\cbar,r)$.  Write the tuple $\nbar_*\otimes(\ku\oplus\kn)$ as $\nbar\in\N^k$ where, for $i\not\in I$ we have $n_i=n_{*,i}$ and, for $i\in P\in\cQ$, $n_i=n_P+u_{P,i}$. For each $P\in\cQ$, since $n_P\in U(P,\ubar_P)$, we may let $m_P$ be the common value of $f(n_P+u_{P,i})$ for $i\in P$. In particular, $a_{n_P+u_{P,i}}=b_{m_P}+s_i$ for all $i\in P\in\cQ$. Therefore,
\begin{multline*}
\sum_{i\in [k]}c_ia_{n_i} = \sum_{i\not\in I}c_ia_{n_i}+\sum_{i\in I}c_ia_{n_i}=s+\sum_{P\in\cQ}\sum_{i\in P}c_ia_{n_P+u_{P,i}} =s+\sum_{P\in\cQ}\sum_{i\in P}c_i(b_{m_P}+s_i)\\
=s+\sum_{P\in\cQ}\left(c_Pb_{m_P}+\sum_{i\in P}c_is_i\right)=s+\sum_{P\in\cQ}s_P=r.
\end{multline*}
So $\nbar\in A(\cbar,r)$, as desired. For the left-to-right containment, fix $\nbar\in A(\cbar,r)$. Let $\cP=\cP(\nbar)$ and $\sigma=(r_{n_i})_{i\in I(\cP)}$. We want to show $\nbar\in X(\cP,\sigma)$. Let $\nbar_*=(n_i)_{i\not\in I}$. For each $P\in\cQ$, fix some $i_P\in P$ and let $\kn=(n_{i_P})_{P\in\cP}$. Then, for each $P\in \cQ$, let $\ubar_P=(n_i-n_{i_P})_{i\in P}$. By construction, $n_{i_P}\in U(P,\ubar_P)$ for all $P\in\cQ$. By choice of $K$, $\ku\in E$. By construction, $\nbar=\nbar_*\otimes(\ku\oplus\kn)$. So it remains to verify $\nbar_*\in A_0((c_i)_{i\not\in I},s)$ (where recall $s=r-\sum_{P\in\cQ}\sum_{i\in P}r_{n_i}$). Similar to the equations above, we have $\sum_{i\in I}c_ia_{n_i}=\sum_{P\in\cQ}\sum_{i\in P}r_{n_i}$, and so $\nbar_*\in A_0((c_i)_{i\not\in I},s)$ since $\nbar\in A(\cbar,r)$.
\end{proof}

\begin{lemma}\label{lem:ISfin}
$A_0(\cbar,r)$ is finite for any $k\geq 1$, $\cbar\in\{\nv 1,1\}^k$, and $r\in\Z$.
\end{lemma}
\begin{proof}
Fix $k\geq 1$, $\cbar\in\{\nv 1,1\}^k$, and $r\in\Z$. Define  $C=\{\sum_{i\in I}c_i:I\seq [k]\}$ and $S=\{|c|\lambda_n:n\in\N,~c\in C\backslash\{0\}\}$. Since $C$ is finite we still have that the set $\{\frac{s}{t}:s,t\in S,~t\leq s\}$ is closed and discrete. By Fact \ref{fact:Poon}, we may fix $\epsilon>0$ such that, for any $1\leq l\leq k$, $\nbar\in\N^l$, and $c'_1,\ldots,c'_l\in C$, if $\sum_{i\in I}c'_i\lambda_{n_i}\neq 0$ for all nonempty $I\seq[l]$, then $|c'_1\lambda_{n_1}+\ldots+c'_l\lambda_{n_l}|\geq \epsilon\lambda_{\max\nbar}$.  Let $R=\max\{|x|:x\in F\}$.

For a contradiction, suppose we have an infinite sequence $(\nbar(t))_{t=0}^\infty$ in $A_0(\cbar,r)$. After passing to a subsequence and permuting indices, we may assume without loss of generality that $f(n(t)_1)\leq\ldots\leq f(n(t)_k)$ for all $t\in\N$. In particular, $\sup_{t\in\N}f(n(t)_k)=\infty$. For $m\in\N$, let $\theta_m=b_m-\lambda_m$. Since $|\theta_m|$ is $o(\lambda_m)$, there is $M_0>0$ such that $|\theta_m|\leq \frac{\epsilon}{2k}\lambda_m$ for all $m\geq M_0$. Define
\[
d=\max\{|c_1\theta_{m_1}+\ldots+c_l\theta_{m_l}|:l<k,~m_i<M_0\}.
\]
We may choose $M\geq M_0$ such that $r+kR+d<\frac{\epsilon}{2}\lambda_M$. Choose $t\in\N$ such that $f(n(t)_k)\geq M$ and, for $1\leq i\leq k$, let $m_i=f(n(t)_i)$. Let $\cP=\cP(\nbar(t))$. For each $P\in\cP$, let $m_P$ be the common value of $m_i$ for $i\in P$, and let $c_P=\sum_{i\in P}c_i$. Then $m_P\neq m_{P'}$ for distinct $P,P'\in\cP$, and $c_P\in C\backslash\{0\}$ for all $P\in\cP$ since $\nbar\in A_0(\cbar,r)$. Let $\mbar=(m_P)_{P\in\cP}$ and $\Lambda=\sum_{P\in \cP}c_P\lambda_{m_P}$. Let $Q\in\cP$ be the unique set containing $k$. By construction, $m_Q=\max\mbar\geq  M$. Since $\nbar(t)\in A_0(\cbar,r)$, we have
\[
|\Lambda|=|r-(c_1(r_{n_1}+\theta_{m_1})+\ldots+c_k(r_{n_k}+\theta_{m_k}))|\leq r+kR+d+\textstyle\frac{\epsilon}{2}\lambda_{m_Q}<\epsilon\lambda_{m_Q}.
\]
By choice of $\epsilon$, it follows that $\sum_{P\in X}c_P\lambda_{m_P}=0$ for some nonempty $X\seq\cP$. By $\Q$-linear independence, $c_P=0$ for all $P\in X$, which is a contradiction.
\end{proof}

We can now prove the main result of this section.

\begin{proof}[Proof of Theorem \ref{thm:IS}]
By Lemmas \ref{lem:IStech} and \ref{lem:ISfin}, $A^{\cZ}_0$ is a virtual reduct of $\cN^1_{\ms}$, which is monadically stable of $U$-rank $1$ by Proposition \ref{prop:monstab}. Now apply Proposition \ref{prop:ISss} and Corollary \ref{cor:mainstab}.
\end{proof}

\begin{remark}
The independently sparse sets described  in Example \ref{ex:IS} are all of the form $A=(a_n)_{n=0}^\infty$ such that $|a_n-\lambda_n|$ is $o(\lambda_n)$ for some $\Q$-linearly independent geometric sequence $(\lambda_n)_{n=0}^\infty$. So in this case, we may use $F=\{0\}$ and $f(x)=x$ in the above analysis. It follows that $A_0(\cbar,r)=A(\cbar,r)$ for any $k\geq 1$, $\cbar\in\{\nv1,1\}^k$, and $r\in\Z$. By Lemma \ref{lem:ISfin}, $A^{\cZ}_0$ is interdefinable, via $n\mapsto a_n$, with $\N$ in the language of equality, and thus $A^{\cZ}$ is interdefinable, via $n\mapsto a_n$, with an expansion of $\N$ by unary predicates. 
\end{remark}

%\bibliography{/Users/gabrielconant/Desktop/Math/BibTex/biblio}
\bibliography{/Users/Gabe/Desktop/Math/BibTex/biblio}
\bibliographystyle{amsplain}
\end{document}